\documentclass[12pt]{amsart}
\usepackage{amsmath,amssymb}
\usepackage{amsfonts}
\usepackage{amsthm}
\usepackage{latexsym}
\usepackage{graphicx}

\def\lf{\left}
\def\ri{\right}

\def\p{\partial}

\def\R{\mathbb{R}}
\def\C{\mathbb{C}}

\def\vv<#1>{\langle#1\rangle}

\def\XXint#1#2{\setbox0=\hbox{$#1{#2}{\int}$}{#2}\kern-.5\wd0 }

\def\XXint#1#2#3{{\setbox0=\hbox{$#1{#2#3}{\int}$}
     \vcenter{\hbox{$#2#3$}}\kern-.5\wd0}}

\def\vv<#1>{\langle#1\rangle}
\newtheorem{thm}{Theorem}[section]

\newtheorem{lem}{Lemma}[section]

\theoremstyle{definition}

\theoremstyle{remark}

\newtheorem{rem}{Remark}[section]

\numberwithin{equation}{section}

\begin{document}

\title{Complex product manifolds and bounds of curvature}


\begin{abstract} Let $M=X\times Y$ be the product of two complex manifolds of positive dimensions.
 In this paper, we prove that there is no complete K\"ahler metric $g$ on $M$ such that: either (i) the holomorphic bisectional curvature of $g$ is bounded by a negative constant and the Ricci curvature is bounded below by $-C(1+r^2)$ where $r$ is the distance from a fixed point; or (ii) $g$ has nonpositive sectional curvature and the holomorphic bisectional curvature is bounded above by $-B(1+r^2)^{-\delta}$ and the Ricci curvature is bounded below by $-A(1+r^2)^\gamma$ where $A, B, \gamma, \delta$ are positive constants with $\gamma+2\delta<1$. These are generalizations of some previous results, in particular the result of Seshadri and Zheng \cite{SeshadriZheng}.
\end{abstract}

\renewcommand{\subjclassname}{\textup{2000} Mathematics Subject Classification}
\author{Luen-Fai Tam$^1$ and Chengjie Yu}

\address{The Institute of Mathematical Sciences and Department of
 Mathematics, The Chinese University of Hong Kong,
Shatin, Hong Kong, China.} \email{lftam@math.cuhk.edu.hk}

\address{Department of Mathematics, Shantou University, Shantou, Guangdong, China } \email{cjyu@stu.edu.cn}

\thanks{$^1$Research partially supported by Hong Kong RGC General Research Fund  \#GRF 2160357}
\date{September 2009}

\maketitle
\markboth{Luen-Fai Tam and Chengjie Yu}
 {Complex product manifolds and bounds of curvature}
\section{Introduction}

In \cite{SeshadriZheng}, Seshadri and Zheng proved the following result:

\begin{thm}\label{SeshadriZheng} Let $M=X\times Y$ be the product of two complex manifolds of positive dimensions. Then $M$ does not admit any complete Hermitian metric with bounded torsion and holomorphic bisectional curvature bounded between two negative constants.
\end{thm}
In particular, there is no complete K\"ahler metric on $M$ with holomorphic bisectional curvature bounded between two negative constants. For earlier results in this direction see \cite{Yang,Zheng93,Zheng94,Seshadri2006}. The result of  Seshadri-Zheng has been generalized by Tosatti \cite{Tosatti} to almost-Hermitian manifolds.

On the other hand, there is an open question   whether the assumption on the lower bound of the curvature can be removed. In fact, it is an open question raised by N. Mok (see \cite{SeshadriZheng}) whether the bidisc admit a complete K\"ahler metric with bisectional curvature bounded above by -1.

In this work, by using a local version of the generalized Schwartz lemma of Yau \cite{Yau-Schwartz} and a Omori-Yau type maximum principle of Takegoshi \cite{Takegoshi}, we prove the following:

\begin{thm}\label{th-1}
Let $M=X^m\times Y^n$ be the product of two complex manifolds with
positive dimension. Then, there is no complete K\"ahler metric on
$M$ with Ricci curvature $\geq -A(1+r)^2$ and holomorphic
bisectional curvature $\leq -B$, where $A$ is some nonnegative
constants, $B$ is some positive constant, and $r(x,y)=d(o,(x,y))$ is
the distance of $(x,y)$ and a fixed point $o\in M$.
\end{thm}

On the other hand, Seshadri \cite{Seshadri2006} has constructed a complete K\"ahler metric on $\C^n$ with negative curvature. It seems that
the assumption on the upper bound of the curvature in Theorem \ref{SeshadriZheng} is necessary. However, one can also relax the upper bound of the curvature as follows:

\begin{thm}\label{th-2}
Let $M=X^m\times Y^n$ be the product of two complex manifolds with
positive dimensions. Then, there is no complete K\"ahler metric on
$M$ with Ricci curvature $\geq -A(1+r^2)^\gamma$, holomorphic
bisectional curvature $\leq -B(1+r^2)^{-\delta}$, and nonpositive
sectional curvature, where $\gamma\geq 0$, $\delta>0$ such that
$\gamma+2\delta<1$, $A,B$ are some positive constants, and
$r(x,y)=d(o,(x,y))$ is the distance of $(x,y)$ and a fixed point
$o\in M$.
\end{thm}
Our method   relies on a simple observation. Suppose there is a
Hermitian metric $g$ with negative holomorphic bisectional curvature
on $M=X\times Y$. Let $q$ be any fixed point in $Y$. Then the
holomorphic vector bundle $V$ over $X_q:=X\times \{q\}$, with fibre
$V_{(x,q)}=T^{1,0}_qY\subset T_{(x,q)}^{1,0}M=T^{1,0}_xX\oplus
T^{1,0}_qY$, as a subbundle of $T^{1,0}M|_{X_q}$, is negative.
However, $V=T^{(1,0)}_qY\times X_q$ is a trivial vector bundle, and
hence have nonzero global holomorphic sections. Hence the question
reduces to the question on the existence of nontrivial global
holomorphic sections on the negative vector bundle $V$ . When $X$ is
compact, then one can conclude easily that this is impossible (See
Kobayashi-Wu \cite{Wu-Kobayashi}). Hence $X\times Y$ does not have a
K\"ahler metric with negative holomorphic bisectional curvature. The
result was first proved by Zheng \cite{Zheng94} using different
method (When $g$ is K\"ahler, it is first proved by Yang
\cite{Yang-2}). Here the metric is not even assumed to be complete.

When $X$ is noncompact and curvature bounds are relaxed, we can only
have vanishing theorem with some restriction on the growth of the
global section. Controlling the growth of the global section can be
done by a local version of Schwartz lemma. Moreover, we need some
geometric condition on $X_q$, for example the validity of Omori-Yau
type maximum principle on $X_q$. This is guaranteed by a volume
estimate of $X_q$ and a theorem of Takegoshi \cite{Takegoshi}.

The paper is organized as follows: In section 2, we will prove Theorem \ref{th-1} and in section 3 we will prove Theorem \ref{th-2}.

\section{Proof  of Theorem \ref{th-1}}

Before we prove the theorem, we need several lemmas. First, we have the following local version of the  Schwartz lemma by Yau \cite{Yau-Schwartz}. See also \cite[Theorem 2.1]{ChenChengLu79}.
\begin{lem}\label{lem-schwartz}
Let $(M^m,g)$ and $(N^n,h)$ be  two complete K\"ahler manifolds and
let $f$ be  a holomorphic map from $M$ to $N$. Let $o\in M$ and let
$R>0$. Suppose the Ricci curvature of $B_o(2R)$ is bounded from
below by $-K$ and the holomorphic bisectional curvature at every
point in $f(B_o(2R))$ is bounded above by $-B$ where $K$ and $B$ are
positive constant. Then on $B_o(R)$,
\begin{equation}\label{Schwarz-e1}
f^*\omega_h\leq C\cdot\frac{K+R^{-2}\lf(1+
K^\frac12R\coth(K^\frac12R)\ri)}{B}\omega_g
\end{equation}
for some positive constant $C$ depending only on $m$.
\end{lem}
\begin{proof}
Let $u=\|\partial f\|^2$ which is half the energy density
of $f$. It is clear that $f^*\omega_h\leq \|\partial
f\|^2\omega_g$.  Then $u$ satisfies the following
 inequality on $B(2R)$: (See \cite{Chern}, \cite{Yau-Schwartz})
\begin{equation}\label{eqn-energy-density}
\frac{1}{2}\Delta u\geq -Ku+Bu^2.
\end{equation}

Let $\eta\ge0$ be a smooth function on $\R$ such that (1)
$\eta(t)=1$ for $t\leq 1$, (2) $\eta(t)=0$ for $t\geq 2$,
(3) $-C_1\leq \eta'/\eta^\frac12\leq 0$ for all $t\in\R$,
and (4) $|\eta''(t)|\leq C_1$ for all $t\in \R$ for some
absolute constant $C_1>0$. Let $\phi=\eta(r/R)$.

Suppose $\phi u$ attains maximum at $\bar x\in B_o(2R)$.   We
can assume that $\phi(  \bar x)>0$ otherwise $u(\bar x)=0$ for any
$x\in B_o(R)$ and we are done. Using an argument of Calabi as in \cite{ChengYau75}, we may assume that $\phi u$ is smooth at $\bar x$. Then, we have (1) $\nabla
(\phi u)(  \bar x)=0$ which implies that at $\bar x$, $\nabla u
=-u\phi^{-1}\nabla\phi  $, (2) $\Delta(\phi u)(\bar x)\leq
0$. Hence at $\bar x$, we have:
\begin{equation}\label{eqn-phi-u}
\begin{split}
0\ge&\Delta(\phi u)\\
=&\phi \Delta u+2\vv<\nabla \phi,\nabla u>+u\Delta\phi \\
=&\phi\Delta u+2\vv<\nabla \phi ,-u\phi^{-1}\nabla\phi>+ u \Delta\phi\\
= & \phi \Delta u-2uR^{-2}\lf|\frac{(\eta')^2}{\eta}\ri| +
u\lf(R^{-1}\eta'\Delta r+ R^{-2}\eta''\ri)
 \\
\ge &\phi(-2Ku+2Bu^2)-C_2R^{-2}\lf(1+
K^\frac12R\coth(K^\frac12R)\ri)u.
\end{split}
\end{equation}
where $C_2$ is   a positive constant depending only on $m$.
Here we have used \eqref{eqn-energy-density}, the
properties of $\eta$ and the Laplacian comparison
\cite{Wu}. Hence
$$
2B(\phi u)^2(\bar x)\le \lf[2K+C_2R^{-2}\lf(1+
K^\frac12R\coth(K^\frac12R)\ri)\ri](\phi u)(\bar x).
$$
Hence
$$
\sup_{B_o(R)} u\le \sup_{B_o(2R)} \phi u=(\phi u)(\bar x)\le
\frac{2K+C_2R^{-2}\lf(1+
K^\frac12R\coth(K^\frac12R)\ri)}{2B}.
$$
From this the lemma follows.

\end{proof}

Before we state the next lemma, let us introduce some notations. Let
$M=X\times Y$ and let $o=(p,q)\in X\times Y$ be a fixed point.   For
any $x\in X$, let $Y_x=\{x\}\times Y$ with induced metric denoted as
$g^x$ and for any $y\in Y$, let $X_y=X\times\{y\}$ with induced
metric denoted as $g^y$.

\begin{lem}\label{volume-l1} Let $M$, $X, Y$ as in Theorem \ref{th-1}. Suppose there is a complete K\"ahler metric on $M$ with Ricci curvature bounded from below by $-A(1+r)^2$ and with holomorphic bisectional curvature bounded from above by $-B<0$.  Let $o=(p,q)\in X\times Y$ be a fixed point. Let $V^{^{X_q}}_p(r)$ be the volume of the geodesic ball of radius $r$ with center at $p$ with respect to the induced metric $g^q$. Then
$$
V^{^{X_q}}_p(r)\le C_1\exp(C_2 r^2)
$$
for some constants $C_1$ and $C_2$ independent of $r$.
\end{lem}
\begin{proof} Let $x_0\in X$ be any point. Consider the projection $\pi''_{x_0}:X\times Y\to Y_{x_0}$ such that $\pi''_{x_0}(x,y)=(x_0,y)$. Note that the holomorphic bisectional curvature of   $Y_{x_0}$ is still bounded above by $-B$. By Lemma \ref{lem-schwartz}, there is a constant $C_1$ independent of $x_0$ such that
\begin{equation}\label{projection-e1}
(\pi''_{x_0})^*(g^{ x_0})|_{(x,y)}\le C_1\lf(1+r (x,y)\ri)^2g|_{(x,y)}
\end{equation}
 for all   $(x,y)$. Similarly, if we choose $C_1$ large enough, we also have:
 \begin{equation}\label{projection-e2}
(\pi'_{y_0})^*(g^{y_0})|_{(x,y)}\le C_1\lf(1+r (x,y)\ri)^2g|_{(x,y)}
\end{equation}
  for any $y_0\in Y$, and $(x,y)\in M$ where  $\pi'_{y_0}$ is the projection from $M$ onto $X_{y_0}$.

  Let $\gamma$ be  any smooth curve in  $Y_{p} $ from $(p,q)$. Then by \eqref{projection-e1},  for any $(x,q)\in X_q$ with $r(x,q)\le R$, the length   $L(\pi_{x}\circ \gamma)$ satisfies:
  $$
  L(\pi_{x}\circ \gamma)\le C_1^\frac12 \lf(1+r (x,q)\ri)L(\gamma).
  $$
  Hence there is $\rho>0$ such that  for $R>1$ if $B^p(\rho)$ is the geodesic ball in $Y_p$ with radius $\rho$ and center at $(p,q)$, then   $\pi_{x}(B^p(\rho))\subset B_o(2R)$ for all $(x,q)\in B_o(R)$.

  On the other hand, by \eqref{projection-e2}, the Jacobian $J(\pi'
  _q)$ of $\pi'_q$ at $(x,y)$ satisfies:
  \begin{equation}\label{Jacobian-e1}
     J(\pi'_q)(x,y)\le C_2 \lf(1+r(x,y)\ri)^{2m}
  \end{equation}
  for some constant $C_2$ for all $(x,y)$.

Now, let   $R>1$ be any constant. Let $dV_x$ be the volume element of $Y_x$ and $dV_y$ be the volume element of $X_y$.  By the  co-area formula (see \cite{BZ}, we have
\begin{equation}\label{eqn-v}
\begin{split}
V_o(2R)=&\int_{M}\chi_{B_o(2R)}dV_g\\
=&\int_{X_q}\int_{y\in Y_x}\chi_{B_o(2R)}|J(\pi'_q)|^{-1}(x,y)dV_x dV_q\\
\geq&C_2^{-1}(1+R)^{-2m}\int_{(x,q)\in B_o(R)}\int_{\pi_x''(B^p(\rho))}dV_xdV_q\\
\geq&C_3 (1+R)^{-2m}(1+R)^{-2n}V^p(\rho)V^q(R)
\end{split}
\end{equation}
for some constant $C_3>0$ for all $R$ by \eqref{projection-e1}. Here $V^p(\rho)$ is the volume of $B^p(\rho)$ in $Y_p$ and $V^q(R)$ is the volume of the geodesic ball in $X_q$ with radius $R$ and center at $(p,q)$.

By volume comparison, we have $V_o(2R)\le \exp(C(1+R)^2)$ for some constant $C$. From this and \eqref{eqn-v}, the result follows.
\end{proof}

We also need the following result of Takegoshi \cite[Theorem 1.1]{Takegoshi}:
\begin{thm}\label{Takegoshi} Let $M$ be a complete noncompact Riemannian manifold. Suppose there is a smooth function $f$ such that $S=\{f>\delta\}$ is nonempty for some $\delta>0$ and on $S$
$$
\Delta f\ge \frac{Cf^{1+a}}{(1+r)^b}
$$
for some positive constants $C, a$ and $0\le b<2$ where $r$ is the distance function from some fixed point. Then the volume $V(r)$ of the geodesic ball with radius $r$ satisfies:
$$
\liminf_{r\to\infty}\frac{\log V(r)}{r^{2-b}}=\infty.
$$

\end{thm}

\begin{proof}[Proof of Theorem \ref{th-1}]
We proceed by contradiction. Let $g$ be a complete K\"ahler metric
on $M$ satisfying the assumptions.

Suppose $o=(p,q)\in M$. Let $u$ be vector in
$T_{p,q}^{1,0}(M)=T_{p}^{1,0}(X)\times T_{q}^{1,0}(Y)$ such that
$u\in \{0\}\times T_{q}^{1,0}(Y)$ and such that $g(u,\bar u)=1$. Let
\begin{equation*}
f(x)=f(x,q)=\|(\pi''_{x})_*(u)\|^2.
\end{equation*}
Then $f$ is a function on $X_q$.
Then, by \eqref{projection-e1}
\begin{equation}\label{fbound-e1}
f(x)\leq C_1(1+r(p,q))^2 g_{u\bar u}(p,q)=C_1g_{u\bar u}(p,q).
\end{equation}
Hence, $f$ is a positive bounded function.

Let $(z^1,z^2,\cdots,z^m,z^{m+1},\cdots,z^{n+m})$ be a holomorphic
coordinate of $M$ at $(x,q)$ such that (1) $(z^1,z^2,\cdots,z^m)$ is
a normal coordinate of $X_q$ at $x$ and (2) $g_{a\bar
b}(x,q)=\delta_{ab}$, $m+1\le a, b\le m+n$. Then, by identifying $(\pi_x)_*(u)$ with $u$, we have
\begin{equation}\label{eqn-f}
\begin{split}
\Delta_{X_q}f(x)=&2\sum_{i=1}^m\partial_{i}\partial_{\bar i}g_{u\bar u}(x,q)\\
=&2\sum_{i=1}^m(-R_{u\bar ui\bar i}+g^{\bar ba}\partial_{i}g_{u\bar b}\partial_{\bar i}g_{a\bar u})(x,q)\\
=&-2\sum_{i=1}^{m}R_{u\bar ui\bar i}(x,q)+2\sum_{i=1}^m\sum_{b=1}^{n+m}|\partial_{i}g_{u\bar b}|^2(x,q)\\
\geq& 2mBg_{u\bar u}(x,q)\\
=&2mBf(x).
\end{split}
\end{equation}
Combining this with \eqref{fbound-e1}, we have
$$
\Delta_{X_q}f\ge \frac{2mB}{C_1}f^2.
$$
By Lemma \ref{volume-l1} and Theorem \ref{Takegoshi}, we have a contradiction because $f>0$.

\end{proof}

 \section{Proof of Theorem \ref{th-2}}

In order to prove the second main result, we need the following
lemma. We will use the notations as in the previous section.
\begin{lem}\label{lem-schwartz-2}
Let $M=X^m\times Y^n$ be the product of two simply connected complex
manifolds with positive dimension. Suppose that $g$ is a complete
K\"ahler metric on $M$ with Ricci curvature  $\geq
-A(1+r^2)^\gamma$, holomorphic bisectional curvature $\leq
-B(1+r^2)^{-\delta}$, and nonpositive sectional curvature, where
$\gamma\geq 0$, $\delta>0$ such that $\gamma+2\delta<1$, $A,B$ are
some positive constants, and $r=r(x,y)  $ is the distance of
$(x,y)\in X\times Y$ and a fixed point $o\in M$. Then, there is a
positive constant $C$ depending only on $m, n, \gamma,\delta, A$ and
$B$, such that for any $x_0\in X$,
\begin{equation}\label{Schwarz-2-e1}
(\pi''_{x_0})^*(g^{x_0})|_{(x,y)} \leq C(1+r^2(x,y))^\gamma(1+r^2(x_0,y))^\delta g|_{(x,y)}
\end{equation}
for any $(x,y)\in M$ and $y\in Y$.
\end{lem}
\begin{proof}
For any point $x_0\in X$, let $u=\|\p \pi_{x_0}''\|^2.$ Then as
before, by \cite{Chern}, \cite{Yau-Schwartz}, we have:

\begin{equation}\label{eqn-Q}
\Delta u(x,y)\geq -2A(1+r^2(x,y))^{\gamma}u(x,y)+2B(1+r^2(x_0,y))^{-\delta}u(x,y)^{2}.
\end{equation}
Let $v(x,y)=r^2(x_0,y)$. Since $M$ is simply connected and has nonpositive curvature, $r^2(x,y)$ and $v$ are both smooth functions. In the following $\alpha, \beta$ range from $m+1$ to $m+n$. For $(x,y)\in M$, let $z^1,z^2,\cdots,z^{m}$ be holomorphic coordinates of $x$ in $X$ and
 $z^{m+1},\cdots,z^{m+n} $ be   holomorphic
coordinates of $y$ in $Y$ such that   (1) $g_{\alpha\bar \beta}(x,y)=\delta_{\alpha  \beta}$, $m+1\le \alpha  \beta\le m+n$ and (2) $g_{\alpha\bar \beta}(x_0,y)=\lambda_\alpha\delta_{\alpha\beta}$. Here $(z^{m+1},\cdots,z^{m+n})$ are also   considered as holomorphic coordinates of $Y_{x_0}$ because $\pi''_{x_0}$ is a biholomorphism between $Y_{x}$ and $Y_{x_0}$. Then
$u(x,y)=g^{\bar \beta\alpha}(x,y)g_{\alpha\bar \beta}(x_0,y).$
 Then
\begin{equation}\label{eqn-gradient}
\begin{split}
&\|\nabla v(x,y)\|^2(x,y)\\
=&2g^{\bar ba}(x,y)v_a(x,y) v_{\bar b} (x,y)\\
=&2g^{\bar \beta\alpha}(x,y)v_\alpha(x,y)v_{\bar \beta}(x,y)\\
=&2u(x,y)g^{\bar \beta\alpha}(x_0,y)v_\alpha(x,y)v_{\bar \beta}(x,y)\\
\leq& 4u(x,y)v(x,y)
\end{split}
\end{equation}
where we have used the fact that the   $|\nabla r (x,y)|=1$ and $r(x_0,y)=r(x,y)|_{Y_{x_0}}$.
On the other hand, since $r^2$ is convex, we have
\begin{equation}\label{eqn-laplacian}
\begin{split}
&\Delta v(x,y)\\
=&2g^{\bar ba}(x,y)v_{a\bar b}(x,y)\\
=&2g^{\bar\beta\alpha}(x,y)v_{\alpha\bar\beta}(x_0,y)\\
=&2 \sum_{\alpha}v_{\alpha\bar\alpha}(x_0,y)\\
\leq& 2u(x,y) \sum_{\alpha}\lambda_\alpha^{-1}v_{\alpha\bar\alpha}(x_0,y)\ \text{\ \ (since $v_{\alpha\bar\alpha}>0$)}\\
=&u(x,y)\Delta_{Y_{x_0}} v(x_0,y)\\
\leq& u(x,y)(\Delta r^2)(x_0,y)\text{\ \ (since $r^2$ is convex)}\\
\leq&C_1u(x,y)(1+v(x,y))^\frac{\gamma+1}{2}
\end{split}
\end{equation}
for some constant $C_1$ by the Laplacian comparison \cite{Wu} and the assumption on the Ricci curvature of $M$. Here and below $C_i$ will denote constants depending only on $m, n, \gamma, \delta, A, B$.
Let $$w(x,y)=u(x,y)(C_2+v(x,y))^{-\delta}$$
where $C_2>1$ is
a constant to be determined later.   Then,
\begin{equation}\label{eqn-tilde-Q}
\begin{split}
 \Delta   w
=&(C_2+ v)^{-\delta}\Delta u -2\delta (C_2+v)^{-1-\delta}\vv<\nabla u,\nabla v>\\
&-u\delta(C_2+v)^{-1-\delta}\Delta v
+u\delta(\delta+1)(C_2+v)^{-2-\delta}|\nabla v|^2\\
\ge& \lf(2B-\delta C_1\lf(C_2+v\ri)^{\frac\gamma2+\delta-\frac12}\ri)w^2-2A(1+r^2)^\gamma w-2\delta (C_2+v)^{-1}\vv<\nabla w,\nabla v>\\
\ge& \lf(2B-\delta C_1 C_2 ^{\frac\gamma2+\delta-\frac12}\ri)w^2-2A(1+r^2)^\gamma w-4\delta |\nabla w| w^\frac12
\end{split}
\end{equation}
where we have used \eqref{eqn-gradient}, \eqref{eqn-laplacian}, \eqref{eqn-Q} and  the that $C_2>1$ and $\gamma+2\delta<1$. Hence we may choose $C_2>0$ large enough, so that
 \begin{equation}\label{eqn-tilde-Q-final}
\Delta w\ge C_3 w^2-2A(1+r^2)^\gamma w-4\delta |\nabla w| w^\frac12
\end{equation}
for some $C_3>0$.  Then one can proceed as in the proof of Lemma \ref{lem-schwartz} to conclude that \eqref{Schwarz-2-e1} is true.

\end{proof}

\begin{proof}[Proof of Theorem \ref{th-2}]
  First observe that we may assume $M$ is simply connected because the distance function in the universal cover of $M$ is greater than or equal to the distance function in $M$. Suppose there is complete K\"ahler metric $g$ on $M$
on $M$ satisfying the curvature  assumptions.

Let $o=(p,q)$. As in the proof of Theorem \ref{th-1}, let $u$ be a vector in $T_{p,q}^{1,0}(M)=T_{p}^{1,0}(X)\times T_{q}^{1,0}(Y)$ such that $u\in \{0\}\times T_{q}^{1,0}(Y)$ and such that $g(u,\bar u)=1$. Let
\begin{equation*}
f(x)=f(x,q)=\|(\pi''_{x})_*(u)\|^2.
\end{equation*}
Then $f(x)$ is a function on $X_q$. By the same computation as in \eqref{eqn-f},
\begin{equation}\label{eqn-f-2}
\Delta_{X_q} f(x)\geq 2mB(1+r^2(x,q))^{-\delta}f(x).
\end{equation}

By Lemma \ref{lem-schwartz-2}, we have
\begin{equation}\label{eqn-growth-f}
\begin{split}
0<f(x)=&\|(\pi''_{x})_*(u)\|^2\\
=& (\pi_x'')^*(g^x)(u,\bar u)\\
\le &C_1(1+r^2(p,q))^\gamma(1+r^2(x,q))^\delta g(u,\bar u)\\
= &C_1(1+r^2(x,q))^\delta,
\end{split}
\end{equation}
where $C_1$ is a constant independent of $x$. We may proceed as in the proof of Theorem \ref{th-1} to estimate the volume growth of $X_q$ and use Theorem \ref{Takegoshi} to finish the proof. However, since the curvature of $M$ is nonpositive, we may proceed in a more simple way.

Let $h(x)=\log f(x)-2\delta\log(C_2+r^2(x,q))$ where $C_2>1$ is some
constant to be determined. By \eqref{eqn-growth-f}, $h$ achieves its
maximum at some point $(\bar x,q)\in X_q$. Then at $(\bar x,q)$
\begin{equation}\label{eqn-gradient-h=0}
\nabla_{X_q} \log f(\bar x)=2\delta\nabla_{X_q}\log(C_2+r^2(\bar x,q))\ \ \text{and}\ \    \Delta_{X_q}h\le0.
\end{equation}
Since $r^2$ is convex, we have $|\nabla_{X_q}r(x,q)|\le 1$ and $\Delta_{X_q}r^2(x,q)\le \Delta r^2(x,q)\le C_3(1+r^2(x,q))^{\frac{1+\gamma}2} $ for some constant $C_3$ independent of $x$. Let $r=r(\bar x, q)$, then at $(\bar x,q) $, using \eqref{eqn-f-2} and the fact that $\gamma+2\delta<1$, we have
\begin{equation}\label{eqn-laplacian-h}
\begin{split}
0\geq&\Delta_{X_q}h(\bar x)\\
=&f^{-1}\Delta_{X_q} f -(2\delta)^2|\nabla_{X_q}\log(C_2+r^2(\bar x,q))|^2-2\delta(C_2+r^2)^{-1}\Delta_{X_q}r^2(\bar x,q)\\
&+2\delta|\nabla_{X_q} \log(C_2+r^2(\bar x,q))|^2\\
\ge& 2mB(1+r^2)^{-\delta}-2\delta C_3(C_2+r^2)^{\frac{-1+\gamma}2}
\\
>& 2(C_2+r^2)^{-\delta}\lf(mB-\delta C_3(C_2+r^2)^{\frac{-1+\gamma+2\delta}2}\ri)\\
>& 2(C_2+r^2)^{-\delta}\lf(mB-\delta C_3 C_2^{\frac{-1+\gamma+2\delta}2}\ri)\\
>&0,
\end{split}
\end{equation}
if we chose $C_2>1$ large enough, such that $ \delta
 C_3C_2^{-\frac{1-\gamma-2\delta}{2}}< mB$. This can be done because $\gamma+2\delta<1$.  Hence we have a contradiction. This completes the proof of the theorem.
\end{proof}
\begin{rem}
Letting $\gamma=0$ in Theorem \ref{th-2}, we know that there is no
complete K\"ahler metric on $X\times Y$ with Ricci curvature bounded
from below and sectional curvature $\leq -A(1+r^2)^{-\delta}$ for
any $\delta<\frac{1}{2}$. We may ask the problem if $\frac{1}{2}$ is
the optimal power.

In \cite{Wu}, Greene-Wu proved that if a Hermitian manifold $M$ has
holomorphic sectional curvature $\leq -A(1+r^2)^{-1}$, then $M$ is
hyperbolic in the sense of Kobayashi-Royden. Note that $\C^n$ is not
hyperbolic in the sense of Kobayashi-Royden. So, there is no
Hermitian metric on $\C^n$ with holomorphic sectional curvature
$\leq -A(1+r^2)^{-1}$. On the other hand, the example given by
Seshadri \cite{Seshadri2006} has holomorphic bisectional curvature
$\leq -A[(1+r^2)\log(2+r)]^{-1}$. Therefore the optimal power must
be in [1/2,1].
\end{rem}

\end{document}